\documentclass[review]{elsarticle}%\documentclass{elsarticle}
\usepackage{booktabs}
\usepackage{subfigure}
\usepackage{multirow}
\usepackage{tipa}
\journal{a journal}

\usepackage{amsfonts}
\usepackage{amssymb}
\usepackage{amsmath}
\usepackage{extarrows}
\usepackage{amsthm}
\numberwithin{equation}{section}
\usepackage{amsbsy}
\usepackage{graphicx}
\usepackage{color}
\newtheorem{proposition}{Proposition}[section]
\newtheorem{theorem}{Theorem}[section]

\newtheorem{remark}{Remark}[section]
\newtheorem{definition}{Definition}[section]
\allowdisplaybreaks

\begin{document}
\begin{frontmatter}
		%% 所有作者一起写
\title{On De la Pe\~{n}a Type  Inequalities for Point Processes}

%\tnotetext[t2]{xx.}

\author{Naiqi Liu
			%\fnref{fn2}
}
\ead{18678896336@163.com}

		%\fntext[fn1]{xx.}
		%\fntext[fn2]{xx.}
		%\fntext[fn3]{xx.}
\address{School of Mathematics, Shandong University, Jinan, China}

\author{Vladimir V.   Ulyanov
			%\fnref{fn2}
}
\ead{vulyanov@cs.msu.ru}

		%\fntext[fn1]{xx.}
		%\fntext[fn2]{xx.}
		%\fntext[fn3]{xx.}
\address{Lomonosov Moscow State University, Moscow, Russia; \\
Institute for Financial Studies, Shandong University, Jinan, China}

\author{Hanchao Wang \corref{cor1}
			%\fnref{fn2}
}
\ead{wanghanchao@sdu.edu.cn}

\cortext[cor1]{Corresponding author}
		%\fntext[fn1]{xx.}
		%\fntext[fn2]{xx.}
		%\fntext[fn3]{xx.}
\address{Institute for Financial Studies, Shandong University, Jinan, China}

\begin{abstract}
There has been a renewed interest in exponential  concentration inequalities for stochastic processes in probability and statistics over the last three decades. De la Pe\~{n}a \cite{d}  establishes a nice exponential  inequality for discrete time  locally square integrable martingale . In this paper,  we obtain de la Pe\~{n}a's  inequalities for stochastic integral of multivariate point processes. The proof is primarily based on Dol\'{e}ans-Dade exponential formula and the optional stopping theorem.  As application, we obtain an exponential  inequality for block counting process in $\Lambda-$coalescents. 
\end{abstract}


\end{frontmatter}
%\linenumbers

\section{Introduction}\label{sec-Intro}
\setcounter{equation}{0}
\renewcommand{\theequation}{\thesection.\arabic{equation}}

Let  $S=(S_n)_{n\ge 0}$ be a locally square integrable  martingale on $(\Omega, \mathcal{F}, (\mathcal{F}_n)_{n\ge 1}, \mathbb{P})$.  The predictable quadratic variation of  $S=(S_n)_{n\ge 0}$ is given by 
                    $$<S,S>_n=\sum_{i=1}^{n}\mathbb{E}[((S_i-S_{i-1})^2|\mathcal{F}_{i-1}].$$
Many authors studied the upper bound of 
                              $$\mathbb{P}(S_n\ge x, <S,S>_n\le y).$$
The celebrated Freedman inequality is as follows. 

 \begin{theorem} (Freedman \cite{f}) Let  $S=(S_n)_{n\ge 0}$ be a locally square integrable martingale on $(\Omega, \mathcal{F}, (\mathcal{F}_n)_{n\ge 1}, \mathbb{P}))$. If $|S_k-S_{k-1}|\le c$ for each $1\le k\le n$,  then
            $$\mathbb{P}\big(S_n\ge x,<S,S>_n\le y\big)\le \exp\{-\frac{x^2}{2(y+cx)}\}.$$
\end{theorem}

This  result can be regarded as an  extension of Hoeffding \cite{h}.  Fan, Grama and Liu \cite{fgl1, fgl2}, Rio \cite{r} obtained a series of remarkable extensions of Freedman inequality  \cite{f}. See also  Bercu et al. \cite{bdr} for a recent review in this field. 

 De la Pe\~{n}a \cite{d}  establishes a nice exponential  inequality for  discrete time  locally square integrable martingales.

 \begin{theorem} \label{d1}(De la Pe\~{n}a \cite{d}) Let  $S=(S_n)_{n\ge 0}$ be a locally square integrable and conditionally symmetric martingale on $(\Omega, \mathcal{F}, (\mathcal{F}_n)_{n\ge 1}, \mathbb{P})$.  Then
            $$\mathbb{P}\big(S_n\ge x, \sum_{i=1}^{n}(S_i-S_{i-1})^2\le y\big)\le \exp\{-\frac{x^2}{2y}\}.$$
\end{theorem}

This result is quite different from the classical Freedman's inequality.    The challenge for obtaining Theorem \ref{d1} is to find an approach based on the use of the exponential Markov’s inequality.  De la Pe\~{n}a constructed a supermartingale to get Theorem \ref{d1}.  Furthermore, Bercu and Touati \cite{bta} established the following result for self-normalized martingales, which are similar to Theorem \ref{d1}. 

 \begin{theorem} \label{bt1}(Bercu and Touati \cite{bta} ) Let  $S=(S_n)_{n\ge 0}$ be a locally square integrable martingale on $(\Omega, \mathcal{F}, (\mathcal{F}_n)_{n\ge 1}, \mathbb{P})$.  Then, for all $x,y>0$, $a\ge 0$ and $b>0$, 
            $$\mathbb{P}\big(\frac{|S_n|}{a+b<S,S>_n}\ge x,<S,S>_n\ge\sum_{i=1}^{n}(S_i-S_{i-1})^2+ y\big)\le 2\exp\{-x^2(ab+\frac{b^2y}{2})\}.$$
\end{theorem}
 
It is natural to ask  what will happen when we study the continuous-time processes for the above cases?   Let $(\Omega, \mathcal{F}, (\mathcal{F}_t)_{t\ge 0}, \mathbb{P})$ be a stochastic basis.  $M=(M_t)_{t\ge 0}$ is a continuous locally square integrable martingale.     The predictable quadratic variation of  $M$, $<M,M>$, is a continuous increasing process, such that $(M^2_t-<M,M>_t)_{t\ge 0}$ is a  local martingale.  However, we can not define an analogy for $M$ like $\sum_{i=1}^{n}(S_i-S_{i-1})^2$ in  Theorems \ref{d1} and \ref{bt1}.  Since  $M=(M_t)_{t\ge 0}$ has jumps, we can replace $\sum_{i=1}^{n}(S_i-S_{i-1})^2$  by $\sum_{s\le t}|\triangle M_s|^2$. It is an interesting problem to consider De la Pe\~{n}a type  inequalities for continuous-time local   square integrable martingale with jumps. Some authors obtained the  concentration inequalities for continuous-time stochastic processes. Bernstein's inequality for local martingales with jumps was given by van der Geer \cite{v}. Khoshnevisan \cite{k} found some concentration inequalities for continuous martingales.  Dzhaparidze and van Zanten \cite{dv} extended Khoshnevisan's results to martingales with jumps. 

This paper focuses on the De la Pe\~{n}a type  inequalities for stochastic integrals of multivariate point processes.    Stochastic integrals of multivariate point processes are an essential example of purely discontinuous local martingales.    Let$(\Omega, \mathcal{F}, (\mathcal{F}_t)_{t\ge 0}, \mathbb{P})$  be a stochastic basis.  A stochastic process $M=(M_t)_{t\ge 0}$ is called a purely discontinuous local martingale if $M_0=0$ and $M$ is orthogonal to all continuous local martingales.  The reader is referred to the classic book \cite{js} due to Jacod and Shiryayev  for more information.  We shall restrict ourselves to the integer-valued random measure $\mu$ on $\mathbb{R}_+\times \mathbb{R}$ induced by
a $\mathbb{R}_+\times \mathbb{R}$-valued multivariate point process. In particular, let $(T_{k}, Z_{k}), k\ge 1,$ be a  multivariate point process,  and define
 \begin{equation}
 \mu(dt,dx)=\sum_{k\ge 1}\mathbf{1}_{\{T_{k}<\infty\}}\varepsilon_{(T_{k},Z_{k})}(dt, dx), \label{mp}
 \end{equation}
 where $\varepsilon_{(T_{k},Z_{k})}$ is the delta measure at point $(T_{k},Z_{k})$. Then $\mu(\omega; [0, t]\times \mathbb{R})<\infty$ for all $(\omega, t)\in \Omega\times  \mathbb{R}$.  Let $\tilde{\Omega}=\Omega \times \mathbb{R}_+\times \mathbb{R}$, $\tilde{\mathcal{P}}=\mathcal{P}\otimes \mathcal{B}$, where $\mathcal{B}$ is a Borel $\sigma$-field on $\mathbb{R}$ and $\mathcal{P}$ a $\sigma$-field  generated by all left continuous adapted processes  on $\Omega \times \mathbb{R}_+$. The predictable function is a $\tilde{\mathcal{P}}$-measurable function on $\tilde{\Omega}$.  Let   $\nu$  be the unique predictable compensator of $\mu$ (up to a $ \mathbb{P}$-null set). Namely, $\nu$ is  a  predictable random measure such that for any predictable function $W$, $W\ast  \mu-W\ast\nu$ is a  local martingale, where the $W\ast  \mu$ is defined by
\begin{eqnarray}
W\ast \mu_t \nonumber= \left\{
\begin{array}{ll}
 \int_0^t\int_{\mathbb{R}}W( s,x)\mu(  ds, dx),& \mbox{if} \int_0^t\int_{\mathbb{R}}|W( s,x)|\mu( ds, dx)<\infty, \\
 \\
 +\infty & \mbox{otherwise}.
 \end{array}
 \right. \nonumber\label{StarIn}
\end{eqnarray}
  Note the $\nu$ admits the disintegration
\begin{equation}\label{pp1}
 \nu( dt,dx)= dA_t K(\omega, t; dx),
 \end{equation} 
  where $K(\cdot, \cdot)$ is a  transition kernel from $(\Omega\times \mathbb{R}_+ , \mathcal{P})$ into $(\mathbb{R},\mathcal{B})$, and $A=(A_t)_{t\ge 0}$ is an increasing c\`{a}dl\'{a}g predictable process.   For  $\mu$ in (\ref{mp}), which is defined through  multivariate point process,    $\nu$ admits 
 \begin{equation}
  \nu(dt, dx)=\sum_{n\ge 1}\frac{1}{G_n([t, \infty]\times \mathbb{R})}\mathbf{1}_{\{t\le T_{n+1}\}}G_n(dt, dx),\nonumber
 \end{equation} where $G_n(\omega, ds, dx)$ is a regular version of the conditional distribution of $(T_{n+1}, Z_{n+1})$ with respect to   $\sigma\{T_1, Z_1, \cdots, T_n, Z_n\}$.  In particular, if $F_n(dt)=G_n(dt\times \mathbb{R})$, the point process $N=\sum_{n\ge 1} \mathbf{1}_{[T_n, \infty)}$ has the compensator $A_t=\nu([0, t]\times \mathbb{R})$, which satisfies 
 \begin{equation}
   A_t= \sum_{n\ge 1} \int_{0}^{T_{n+1}\wedge t}\frac{1}{F_n([s, \infty])}F_n(ds).\nonumber
 \end{equation}

Now, we  define the stochastic integrals of multivariate point processes. For a stopping time $T$, $[T]=\{(\omega,t):T(\omega)=t\}$  is the graph of $T$.  For $\mu$ in (\ref{mp}),  define $D=\bigcup_{n=1}^{\infty}[T_n]$.
 With any measurable function $W$ on $\tilde{\Omega}$,  we  define  $a_t=\nu( \{t\}\times \mathbb{R})$, and 
 \begin{eqnarray}
\hat{W}\ast \nu_t \nonumber= \left\{
\begin{array}{ll}
 \int_{\mathbb{R}}W(t,x)\nu( \{t\}\times dx), & \mbox{if}  \int_{\mathbb{R}}|W(t,x)|\nu( \{t\}\times dx)<\infty, \\
 \\
 +\infty & \mbox{otherwise}.
 \end{array}
 \right. \nonumber\label{StarIn}
\end{eqnarray}

 We denote by $G_{loc}(\mu)$ the set of all $\tilde{\mathcal{P}}-$measurable real-valued functions $W$ such that  $[\sum_{s\le t}(\widetilde{W}_s)^2]^{1/2}$ is local integrable variation process, where  $\widetilde{W}_t=W1_{D}(\omega,t)-\hat{W}_t$.

 \begin{definition}
 If $W \in G_{loc}(\mu)$,  the stochastic integral of $W$ with respect to $\mu-\nu$ is defined as a purely discontinuous local martingales, which jump process is indistinguishable from $\widetilde{W}$.  \end{definition}
 
  We denote the stochastic integral of $W$ with respect to $\mu-\nu$  by  $W\ast (\mu-\nu)$.  For a given  predictable function $W$, $W\ast (\mu-\nu)$ is a purely discontinuous local martingale, which is defined  through jump process.  It is easy to prove that $W\ast (\mu-\nu)=W\ast \mu-W\ast \nu$. Denote $M=W\ast (\mu-\nu)$.
  Under some conditions, Wang, Lin and Su \cite{wls} obtained 
                              \begin{equation}
\label{q0}
\mathbb{P}\Big(M_t\ge x, <M,M>_t\le v^2 ~\mbox{for some}~t>0\Big)\le \exp\{-\frac{x^2}{2(v^2+cx)}\}\end{equation}  
  where $<M,M>$ is the  predictable quadratic variation process of $M=W\ast (\mu-\nu)$,  
                       $$<M,M>_t=(W-\hat{W})^2\ast \nu_t+\sum_{1\le s\le t}(1-a_s)\hat{W}_s^2.$$
    When $M$ is a purely discontinuous local martingale, $ \sum_{s\le \cdot}|\Delta M_s|^2-<M,M>$ is a local martingale.  
There will be  an interesting problem when the predictable quadratic variation $<M,M>$ in (\ref{q0}) is replaced by the 
    quadratic variation $ \sum_{s\le \cdot}|\Delta M_s|^2$ .  In this paper, we will estimate the upper bound of two types of tail probabilities:
                       \begin{equation}
\label{q1}
\mathbb{P}\Big(M_t\ge x, \sum_{s\le t}|\Delta M_s|^2\le v^2 ~\mbox{for some}~t>0\Big)\end{equation}
        and  
                       \begin{equation}
\label{q2}
\mathbb{P}\Big(M_t\ge (\alpha+\beta \sum_{s\le t}|\triangle M_s|^2)x, \,\, \sum_{s\le t}|\triangle M_s|^2\ge <M,M>_t+ v^2~\mbox{for some}~t>0\Big ).\end{equation}

It is important to note that the continuity of  $A$  implies the quasi-left continuity of  $M$.  However, the quasi-left continuity  of $M$ can be destroyed easily by changing the filtration in the underlying space.  For example, Let $N$ be a homogeneous Poisson process with respect to  $\mathbb{F}$. Let $(T_n)_{n\ge 0}$ be the sequence of the jump-times of $N$.The process $N$ is not quasi-left continuous in
the filtration $\mathbb{G}$ obtained by enlarging  $\mathbb{F}$ initially with  the $\sigma-$field $\mathcal{R}=\sigma(T_1)$. (
$\sigma_n=(1-\frac{1}{2n})T_1$ is  a sequence of $\mathbb{G}$
-stopping times announcing $T_1$). The main purpose of this paper consists in  estimating  (\ref{q1}) and (\ref{q2}) when $M$ is not quasi-left continuous.

The structure  of the paper is as follows. Section 2 will present our main results  and  give their proofs,  while Section 3 will derive an exponential  inequality for block counting process in $\Lambda-$coalescents as applications.   Usually,  $c,C,K,\cdots$ denote positive constants,  which very often may be different at each occurrence.

\section{ The main results and their proofs}\label{secR}
\setcounter{equation}{0}
\renewcommand{\theequation}{\thesection.\arabic{equation}}

Now, we present our first main result. 

 \begin{theorem}  \label{mr1}Let  $\mu$ be  a multivariate point process, $\nu$ be the predictable  compensator of $\mu$, $a_t=\nu(\{t\}\times \mathbb{R})$, $W$ be a given predictable function  on  $\tilde{\Omega}$, and $W\in  G_{loc}(\mu)$. $M=W\ast (\mu-\nu)$.  Assume $\Delta M\ge -1$. Then for $x>0$, $v>0$, 
   \begin{eqnarray*}\mathbb{P}\Big(M_t\ge x, \, \sum_{s\le t}|\triangle M_s|^2\le v^2\mbox{~for some ~}~t>0\Big)\le    \Big( \frac{v^2+x}{v^2}\Big)^{v^2}e^{-x}.
 \end{eqnarray*}  

\end{theorem}
 
\begin{proof}
 
 For simplicity of notation, put 
 \begin{eqnarray*}
 S(\lambda)_{t}&=&\int_{0}^{t}\int_{\mathbb{R}}(e^{[\lambda (W-\hat{W})-(\lambda+\log(1-\lambda))(W-\hat{W})^2]}-1-\lambda (W-\hat{W}))\nu(ds,dx)\\
 & &+\sum_{s\le t}(1-a_{s})(e^{[-\lambda \hat{W}_{s}+(\lambda+\log(1-\lambda))(\hat{W}_s)^2]}-1+\lambda \hat{W}_{s}),  
   \end{eqnarray*}   
 where $\lambda \in [0,1)$.
 
Furthermore, 
 \begin{eqnarray*}
                 \Delta S(\lambda)_{t}
                &  = & \int_{\mathbb{R}} \Big(e^{[\lambda (W-\hat{W})-(\lambda+\log(1-\lambda))(W-\hat{W})^2]}-1-\lambda (W-\hat{W})\Big)\nu(\{t\}, dx)\\
                & &  +(1-a_t)(e^{[-\lambda \hat{W}_{t}+(\lambda+\log(1-\lambda))(\hat{W}_t)^2]}-1+\lambda \hat{W}_t)\\
                &  = & e^{[-\lambda \hat{W}_{t}-(\lambda+\log(1-\lambda))(\hat{W}_{t})^2]} \Big(\int_{\mathbb{R}} e^{[\lambda W-(\lambda+\log(1-\lambda))(W^2-2W\hat{W})]} \nu(\{t\},dx)+1-a_t\Big) \\
                & &+(1-a_t)(-1+\lambda \hat{W}_t)
                - \int_{\mathbb{R}} \big(1+\lambda (W-\hat{W})\big)\nu(\{t\}, dx)\\
                & = & e^{[-\lambda \hat{W}_{t}-(\lambda+\log(1-\lambda))(\hat{W}_{t})^2]} \Big(\int_{\mathbb{R}} e^{[\lambda W-(\lambda+\log(1-\lambda))(W^2-2W\hat{W})]} \nu(\{t\},dx)+1-a_t\Big)+(1-a_t)\lambda \hat{W}_t\\
                & &  -1 - \lambda \int_{\mathbb{R}} (W-\hat{W})\nu(\{t\}, dx).
                  \nonumber
 \end{eqnarray*}
 In addition, it is easy to see by noting $a_t\le 1$
 \begin{equation}
\int_{\mathbb{R}} e^{[\lambda W-(\lambda+\log(1-\lambda))(W^2-2W\hat{W})]} \nu(\{t\},dx)+1-a_t\ge 0 , \nonumber
  \end{equation}
    and
    \begin{equation}
(1-a_t)\lambda \hat{W}_t=\lambda \int_{\mathbb{R}} (W-\hat{W})\nu(\{t\}, dx). \nonumber
  \end{equation}
In combination,  we have for all $t>0$
    \begin{equation}
    \Delta S(\lambda)_{t} >-1, \nonumber
     \end{equation}
where $\lambda \in [0,1)$. For any semimartingale $S(\lambda)_{t}$,  the Dol\'{e}ans-Dade exponential  is
   \begin{equation}
    \mathcal{E}(S(\lambda))_{t}=e^{S(\lambda)_{t}-S(\lambda)_{0}-\frac{1}{2}<S(\lambda)^{c},S(\lambda)^{c}>_{t}}\prod_{s\le t}(1+\Delta S(\lambda)_{t})e^{-\Delta S(\lambda)_{t}}.\nonumber
    \end{equation}
We shall first show that the process  $\Big(e^{[\lambda M_t-(\lambda+\log(1-\lambda))\sum_{s\le t}(\Delta M_s)^2)]}/\mathcal{E}(S(\lambda))_t\Big)_{t\ge 0}$ is a local martingale.  Denote $X_t=\lambda M_t-(\lambda+\log(1-\lambda))\sum_{s\le t}(\Delta M_s)^2$, $Y_t=\sum_{s\le t}(\Delta M_s)^2$.

The It\^{o} formula yields
        \begin{eqnarray*}
        e^{ X_{t}}&=&1+e^{X_{t-}}\cdot X+\sum_{s\le t}(e^{ X_{s}}-e^{ X_{s-}}-e^{ X_{s-}}\Delta X_{s} )\\
                          &=&1+\lambda e^{X_{t-}}\cdot M-(\lambda+\log(1-\lambda))e^{X_{t-}}\cdot Y\\
                          & &+       \sum_{s\le t}(e^{ X_{s}}-e^{ X_{s-}}-e^{ X_{s-}}\Delta X_{s} ) \\       
                          &=&1+\lambda e^{X_{t-}}\cdot M  + \sum_{s\le t}(e^{ X_{s}}-e^{ X_{s-}}-\lambda e^{ X_{s-}}\Delta M_{s} ). \\      
        \end{eqnarray*}
     For $X$, the jump part of $X$ is
  \begin{eqnarray*}
   \Delta X&=&[\lambda(W-\hat{W})-(\lambda+\log(1-\lambda))(W-\hat{W})^2]1_{D}\\
    & &   -\lambda\hat{W}1_{D^{c}}+(\lambda+\log(1-\lambda))\hat{W}^21_{D^{c}}
   \end{eqnarray*}
where $D$ is the thin set, which is  exhausted by $\{T_n\}_{n\ge 1}$. Thus     \begin{equation}
    \sum_{s\le t}(e^{ \Delta  X_{s}}-1-\lambda \Delta M_{s} )-S(\lambda)=:Z_t
        \end{equation}
  is a local martingale. Furthermore,  we obtain that
     \begin{equation}
     e^{ X}-e^{ X_{-}}\cdot S(\lambda)=1+\lambda e^{X_{t-}}\cdot M + e^{X_{t-}}\cdot Z =:N_1   \end{equation}
     is a local martingale. Following the similar arguments in Wang Lin and Su \cite{wls}, we have
      $\Big(e^{ X_t}/\mathcal{E}(S(\lambda))_t\Big)_{t\ge 0}$
is a local martingale.   In fact,   set $H=e^{X}$, $G=\mathcal{E}(S(\lambda))$, $A=S(\lambda)$ and $f(h,g)=\frac{h}{g}.$  The  It\^{o} formula yields
             \begin{eqnarray*}
             f(H,G)&=&1+\frac{1}{G_{-}}\cdot H-\frac{H_{-}}{G^{2}_{-}}\cdot G\\
             & &+\sum_{s\le \cdot}\Big(\Delta f(H,G)_{s}-\frac{\Delta H_{s}}{G_{s-}}+\frac{f(H,G)_{s-}}{G_{s-}}\Delta G_{s}\Big).
             \end{eqnarray*}
Since  $\mathcal{E}(S(\lambda))=1+\mathcal{E}(S(\lambda))_{-}\cdot S(\lambda)$,  we have 
      \begin{equation}
      \frac{1}{G_{-}}\cdot H-\frac{H_{-}}{G^{2}_{-}}\cdot G=\frac{1}{G_{-}}\cdot N_1.
      \end{equation}
Noting that  $\Delta G=G_{-}\Delta A$, $\Delta N_{1}=\Delta H-H_{-}\Delta A$, we have
          \begin{eqnarray*}
           \Delta f(H,G)_{s}-\frac{\Delta H_{s}}{G_{s-}}+\frac{f(H,G)_{s-}}{G_{s-}}\Delta G_{s} =-\frac{\Delta N_{1s}\Delta A_{s}}{G_{s-}(1+\Delta A_{s})},
               \end{eqnarray*}
where $A$ is a predictable process, and $N$ is a local martingale. By the property of  the Stieltjes integral, we have
     \begin{equation}
     \sum_{s\le \cdot }\Delta f(H,G)_{s}-\frac{\Delta H_{s}}{G_{s-}}+\frac{f(H,G)_{s-}}{G_{s-}}\Delta G_{s}= -\frac{\Delta A}{G_{-}(1+\Delta A)} \cdot N_1.
         \end{equation}
Thus $$\Big(e^{ X}/\mathcal{E}(S(\lambda))\Big)=1+\frac{1}{G_{-}}\cdot N_1-\frac{\Delta A}{G_{-}(1+\Delta A)} \cdot N_1$$
is a local martingale.

 Let
$$
B_1=\{M_t\ge x, \sum_{s\le t}|\triangle M_s|^2\le v^2 \mbox{~for some~} t>0\}
$$
and
$$
\tau_1=\inf\{t>0:  M_t\ge x, \sum_{s\le t}|\triangle M_s|^2\le v^2\}.
$$
Note by (4.12) in  \cite{fgl2}, for $\lambda \in [0,1)$ and   $x\ge -1$,
$$
\exp\{\lambda x+x^2(\lambda+\log (1-\lambda))\}\le 1+\lambda x.
$$
This implies
\begin{equation}\label{2q1}
   \int_0^t\int_{-1}^{ \infty}\exp\{\lambda x+(\lambda+\log(1-\lambda))x^2\}\nu^M(ds, dx)\le \int_0^t\int_{-1}^{ \infty}(1+\lambda x)\nu^M(ds, dx),
\end{equation}
because  $\Delta M_t\ge -1$  for any $t>0$,  where $\nu^M$ is the predictable compensate jump measure of $M$. 
Inequality (\ref{2q1}) implies  $S(\lambda)\le 0$.
Since $e^{x}\ge x+1$ and $e^{S(\lambda)_{t}}\ge \mathcal{E}(S(\lambda)_{t})$, 
     \begin{equation}
     \mathbb{E} \frac{e^{\lambda X_{T}}}{e^{S(\lambda)_{T}}}\le  \mathbb{E}  \frac{e^{\lambda X_{T}}}{\mathcal{E}(S(\lambda))_{T}}=1\end{equation}
    for any stopping time $T$.
Thus  $U=(U_t)_{t\ge 0}$ is a supermartingale, where
$$
U_t=\frac{\exp\{\lambda M_t+(\lambda+\log(1-\lambda))\sum_{s\le t}(\triangle M_s)^2\}}{\exp\{S(\lambda)_t\}}.
$$
Thus,  on $B_1$
$$
U_{\tau_{1}}  \ge \exp\{\lambda x+(\lambda+\log(1-\lambda))v^2\}.
$$
 We have 
  \begin{eqnarray}
\mathbb{P}(B_1) &\le&    \inf_{\lambda\in [0,1)}\exp\{-\lambda x-(\lambda+\log(1-\lambda))v^2\}\nonumber\\
&=&  \Big( \frac{v^2+x}{v^2}\Big)^{v^2}e^{-x}.
\label{3.7}
\end{eqnarray}

\end{proof}

Put 
 \begin{eqnarray*}
 L(\lambda)_{t}&=&\int_{0}^{t}\int_{\mathbb{R}}(e^{[\lambda (W-\hat{W})+f(\lambda)(W-\hat{W})^2]}-1-\lambda (W-\hat{W}))\nu(ds,dx)\\
 & &+\sum_{s\le t}(1-a_{s})(e^{[-\lambda \hat{W}_{s}+f(\lambda)(\hat{W}_s)^2]}-1+\lambda \hat{W}_{s}),  
   \end{eqnarray*}   
where $f(\lambda)\ge 0$ for $\lambda\ge 0$.  We have the following  proposition from the proof of Theorem \ref{mr1}.

 \begin{proposition}  \label{p1}Let  $\mu$ be  a multivariate point process , $\nu$ be the predictable  compensator of $\mu$, $a_t=\nu(\{t\}\times \mathbb{R})$, $W$ be a given predictable function  on  $\tilde{\Omega}$.  $M=W\ast (\mu-\nu)$. Denote $\tilde{X}_t=\lambda M_t-f(\lambda)\sum_{s\le t}(\Delta M_s)^2$,   for $\lambda\ge 0$. Then 
              $e^{ \tilde{X}}/\mathcal{E}(L(\lambda))$ is a local martingale. 
 
 \end{proposition}

In Theorem \ref{mr1}, the condition $\triangle M\ge -1$ plays an important role. In the following theorem, we will present an another result, which is the analogy of Theorem \ref{d1} in continuous time case.

 \begin{theorem}  \label{mr2}Let  $\mu$ be  a multivariate point process, $\nu$ be the predictable  compensator of $\mu$, $a_t=\nu(\{t\}\times \mathbb{R})$, $W$ be a given predictable function  on  $\tilde{\Omega}$, and $W\in  G_{loc}(\mu)$.  $M=W\ast (\mu-\nu)$,  In addition, define
 \begin{eqnarray*}
 \widetilde{S}(\lambda)_{t}&=:&\int_{0}^{t}\int_{\mathbb{R}}(e^{[\lambda (W-\hat{W})-\frac{\lambda^2}{2}(W-\hat{W})^2]}-1-\lambda (W-\hat{W}))\nu(ds,dx)\\
 & &+\sum_{s\le t}(1-a_{s})(e^{[-\lambda \hat{W}_{s}+\frac{\lambda^2}{2}(\hat{W}_s)^2]}-1+\lambda \hat{W}_{s}),  
   \end{eqnarray*}    
 and assume  that for any $t>0$ and  $\lambda>0$, $\widetilde{S}(\lambda)_{t}\le 0$.    Then for $x>0$, $v>0$, 
   \begin{eqnarray*} \mathbb{P}\Big(M_t\ge x, \, \sum_{s\le t}|\triangle M_s|^2\le v^2\mbox{~for some ~}~t>0\Big)\le   \exp{\{-\frac{x^2}{2v^2}\}}
 \end{eqnarray*}  
\end{theorem}

\begin{proof}

Define  $$V_t=\frac{\exp\{\lambda M_t-\frac{\lambda^2}{2}\sum_{s\le t}|\triangle M_s|^2\}}{ \mathcal{E}(\widetilde{S}(\lambda))_{t}}.$$

By Proposition \ref{p1},   $V$ is a local martingale. Note   $\widetilde{S}(\lambda)_{t}\le 0$  for any $t>0$ and  $\lambda>0$. We have 
             \begin{equation}
     \mathbb{E} \frac{\exp\{\lambda M_T-\frac{\lambda^2}{2}\sum_{s\le T}|\triangle M_s|^2\}}{e^{\tilde{S}(\lambda)_{T}}}\le \mathbb{E} V_T=1\end{equation}
   for any stopping time $T$.
   
   Recall that 
$$
B_1=\{M_t\ge x, \sum_{s\le t}|\triangle M_s|^2\le v^2 \mbox{~for some~} t>0\}
$$
and
$$
\tau_1=\inf\{t>0:  M_t\ge x, \sum_{s\le t}|\triangle M_s|^2\le v^2\}.
$$

We have 
 \begin{eqnarray}
\mathbb{P}(B_1) &\le&    \inf_{\lambda \ge 0 }\exp\{-\lambda x+\frac{\lambda^2}{2}v^2\}\nonumber\\
&=& \exp{\{-\frac{x^2}{2v^2}\}}.
\label{2.13}
\end{eqnarray}

\end{proof}

\begin{remark}

 For integrable random variable $\xi$ and a positive number $a>0$, define 
      $$T_a(\xi)=\min(|\xi|, a)sign (\xi).$$
If $\mathbb{E}[\xi]=0$, and for all $a>0$, $\mathbb{E}[T_a(\xi)]\le 0$. Then  $\xi$ is called heavy on left.  Bercu and Touati \cite{btt} extended Theorem \ref{d1} to general case.   Let  $S=(S_n)_{n\ge 0}$ be a locally square integrable on  $(\Omega, \mathcal{F}, (\mathcal{F}_n)_{n\ge 1}, \mathbb{P})$.  If  
     \begin{equation}
\label{hl}
\mathbb{E}[T_a(S_n-S_{n-1})|\mathcal{F}_{n-1}]\le 0\end{equation}
for all $a>0$ and $n>0$, Bercu and Touati \cite{btt} obtained 
            $$\mathbb{P}\big(S_n\ge x, \sum_{i=1}^{n}(S_i-S_{i-1})^2\le y\big)\le \exp\{-\frac{x^2}{2y}\}.$$

In fact, our condition,  $\widetilde{S}(\lambda)_{t}\le 0$,  is analogy of (\ref{hl}) in continuous time case.  Let $N=(N_t)_{t\ge 0}$ be a homogeneous Poisson point process with parameter $\kappa$, and let $(\eta_k)_{ k\ge 1}$ be a sequence of i.i.d. r.v.'s with a common distribution function $F(x)$. Assume $N$ is independent of $(\eta_k)_{ k\ge 1}$.  Define
\begin{equation}
 Y_t=\sum_{k=1}^{N_t}\eta_k, \quad t\ge 0.
 \end{equation}
This is a so-called compound Poisson process. The jump measure of $Y$ is given by
 \begin{equation}
 \mu^{Y}(dt,dx)=\sum_{k\ge 1}\mathbf{1}_{\{T_{k}<\infty\}}\varepsilon_{(T_{k}, \eta_k)}(dt,  dx),
 \end{equation}
and the predictable compensator  $\nu^Y$ is
\begin{equation}
 \nu^{Y}(dt,dx)=\kappa dtF(dx).
 \end{equation}
 Thus $(Y_t-x\ast\nu^Y_t)_{t\ge 0}$ is a purely discontinuous local martingale.  For $(Y_t-x\ast\nu^Y_t)_{t\ge 0}$, 
       \begin{eqnarray*}
 \widetilde{S}(\lambda)_{t}=\kappa\int_{0}^{t}\int_{\mathbb{R}}(e^{[\lambda x-\frac{\lambda^2}{2}x^2]}-1-\lambda x)F(dx)ds.
    \end{eqnarray*}    
  If  $\mathbb{E}[\eta_k]=0$ for any $\kappa\ge 1$,   $\widetilde{S}(\lambda)_{t}\le 0$ implies that 
                 \begin{equation}\label{hl2}
\int_{\mathbb{R}}e^{[\lambda x-\frac{\lambda^2}{2}x^2]}F(dx)\le 1.\end{equation}
   Bercu and Touati \cite{btt} obtained that if $\eta_k$  is heavy on left, then (\ref{hl2}) holds. Thus   our condition  is analogy of (\ref{hl}) in continuous time case.

 \end{remark}

In \cite{d} and \cite{dkl}, there were obtained a series of  exponential inequalities for events involving ratios in the context of  continuous martingales, which in turn extended the results in  \cite{k}.  Su and Wang \cite{sw} extended a similar problem for  purely discontinuous  local martingales in quasi-left continuous case. In this subsection, we obtained the similar inequality for stochastic integrals of multivariate point process.

  \begin{theorem} \label{mr4}Let  $\mu$ be  a multivariate point process,  $\nu$ be the predictable  compensator of $\mu$, $a_t=\nu(\{t\}\times \mathbb{R})$, $W$ be a given predictable function  on  $\tilde{\Omega}$, and $W\in  G_{loc}(\mu)$. Denote $M=W\ast (\mu-\nu)$. Then for all $x\ge 0, \beta>0, $ $v>0$ $\alpha\in \mathbb{R}$,
   \begin{eqnarray*}
&& \mathbb{P}\Big(M_t\ge (\alpha+\beta \sum_{s\le t}|\triangle M_s|^2)x, \,\, \sum_{s\le t}|\triangle M_s|^2\ge  <M,M>_t+v^2\mbox{~for some ~}t>0\Big )\\
&\le &   \exp\{-\frac{x^2}{2}(\alpha \beta +\frac{\beta^2v^2}{2})\}.
\end{eqnarray*}
\end{theorem}

\begin{proof}
Recall that  $V=(V_t)_{t\ge 0}$ is a local martingale, where 
                 $$V_t=\frac{\exp\{\lambda M_t-\frac{\lambda^2}{2}\sum_{s\le t}|\triangle M_s|^2\}}{ \mathcal{E}(\widetilde{S}(\lambda))_{t}}. $$ 
         For any stopping time $T$, 
               
               \begin{equation}
\label{em1}
 \mathbb{E} \frac{\exp\{\lambda M_T-\frac{\lambda^2}{2}\sum_{s\le T}|\triangle M_s|^2\}}{\exp\{\tilde{S}(\lambda)_{T}\}}\le \mathbb{E}[V_T]=1.\end{equation}

               By Markov's inequality, we obtain that for all $\lambda>0$,
\begin{eqnarray*}
         & & \mathbb{P}\Big(M_t\ge (\alpha+\beta \sum_{s\le t}|\triangle M_s|^2)x \mbox{~and~}  \sum_{s\le t}|\triangle M_s|^2\ge <M,M>_t+v^2\mbox{~for some ~}t>0\Big) \\
         & \le &    \mathbb{E}\exp\{\frac{\lambda}{4} M_{\tau_2}- (\frac{\alpha \lambda x}{4}+\frac{\beta \lambda x}{4} \sum_{s\le \tau_2}|\triangle M_s|^2)\} 1_{B_2}\\
         &=&   \exp\{-\frac{\alpha \lambda x}{4}\}  \mathbb{E}\exp\{\frac{\lambda}{4} M_{\tau_2}-\frac{\lambda^2}{8} ( \sum_{s\le \tau_2}|\triangle M_s|^2+<M,M>_{\tau_2})\\
         & &+(\frac{\lambda^2}{8} -\frac{\beta \lambda x}{4})\sum_{s\le \tau_2}|\triangle M_s|^2+\frac{\lambda^2}{8}<M,M>_{\tau_2})\} 1_{B_2}\\
         &\le &  \exp\{-\frac{\alpha \lambda x}{4}\} \sqrt{ \mathbb{E}\exp\{\frac{\lambda}{2} M_{\tau_2}-\frac{\lambda^2}{4}  (\sum_{s\le\tau_2}|\triangle M_s|^2+<M,M>_{\tau_2})\}1_{B_2}}  \\
         & &\times \sqrt{  \mathbb{E}\exp\{(\frac{\lambda^2}{4} -\frac{\beta \lambda x}{2})\sum_{s\le \tau_2}|\triangle M_s|^2+\frac{\lambda^2}{4}<M,M>_{\tau_2}\}1_{B_2}},\end{eqnarray*}
 where 
      $$B_2=\{M_t\ge (\alpha+\beta \sum_{s\le t}|\triangle M_s|^2)x, \sum_{s\le t}|\triangle M_s|^2\ge <M,M> _t+v^2 \mbox{~for~some~} t>0 \},$$
      
      $$
\tau_2=\inf\{t>0:  M_t\ge (\alpha+\beta \sum_{s\le t}|\triangle M_s|^2)x, \sum_{s\le t}|\triangle M_s|^2\ge <M,M>_t +v^2\}.
$$      
In fact,     \begin{eqnarray*}
     & & \mathbb{E}\exp\{\frac{\lambda}{2} M_{\tau_2}-\frac{\lambda^2}{4}  (\sum_{s\le\tau_2}|\triangle M_s|^2+<M,M>_{\tau_2})1_{B_2}\} \\
   &\le &  \sqrt{ \mathbb{E}\frac{\exp\{\lambda M_{\tau_2}-\frac{\lambda^2}{2}  \sum_{s\le \tau_2}|\triangle M_s|^2\}}{\exp\{\widetilde{S}(\lambda)_{\tau_2}\}}1_{B_2}} \sqrt{ \mathbb{E}\exp\{\widetilde{S}(\lambda)_{\tau_2}-\frac{\lambda^2}{2} <M,M>_{\tau_2}\}}.
   \end{eqnarray*}
      
    By (\ref{em1})
        $$ \mathbb{E}\frac{\exp\{\lambda M_{\tau_2}-\frac{\lambda^2}{2}  \sum_{s\le \tau_2}|\triangle M_s|^2\}}{\exp\{\widetilde{S}(\lambda)_{\tau_2}\}}1_{B_2}\le 1.$$  
      
    Furthermore, 
              $$ \mathbb{E}\exp\{\widetilde{S}(\lambda)_{\tau_2}-\frac{\lambda^2}{2} <M,M>_{\tau_2}\}\le 1$$
 by $$\big|\exp\{x-\frac{1}{2}x^2\}-1-x\big|\le \frac{1}{2}x^2, \quad  x\in \mathbb{R}. $$
      
Taking $\lambda=\beta x$,  we get
\begin{eqnarray*}
         \mathbb{P}\Big(B_2\Big) \le    \exp\{-\frac{x^2}{4}(\alpha \beta +\frac{\beta^2v^2}{2})\}\times \sqrt{ \mathbb{P}\Big(B_2\Big)}.
   \end{eqnarray*}   
   Thus
          \begin{eqnarray*}
&& \mathbb{P}\Big(M_t\ge (\alpha+\beta \sum_{s\le t}|\triangle M_s|^2)x, \,\, \sum_{s\le t}|\triangle M_s|^2\ge  <M,M>_t+v^2\mbox{~for some ~}t>0\Big )\\
&\le &   \exp\{-\frac{x^2}{2}(\alpha \beta +\frac{\beta^2v^2}{2})\}.
\end{eqnarray*}

\end{proof}
   
 From the proof of Theorem \ref{mr4}, we can obtain the following results.

  \begin{theorem} \label{mr5}Let  $\mu$ be  a multivariate point process, $\nu$ be the predictable  compensator of $\mu$, $a_t=\nu(\{t\}\times \mathbb{R})$, $W$ be a given predictable function  on  $\tilde{\Omega}$, and $W\in  G_{loc}(\mu)$. Denote $M=W\ast (\mu-\nu)$.  In addition, define
 \begin{eqnarray*}
 \widetilde{S}(\lambda)_{t}&=:&\int_{0}^{t}\int_{\mathbb{R}}(e^{[\lambda (W-\hat{W})-\frac{\lambda^2}{2}(W-\hat{W})^2]}-1-\lambda (W-\hat{W}))\nu(ds,dx)\\
 & &+\sum_{s\le t}(1-a_{s})(e^{[-\lambda \hat{W}_{s}+\frac{\lambda^2}{2}(\hat{W}_s)^2]}-1+\lambda \hat{W}_{s}),  
   \end{eqnarray*}    
 and assume  that for any $t>0$ and  $\lambda>0$, $\widetilde{S}(\lambda)_{t}\le 0$.  Then for all $x\ge 0, \beta>0, v>0$, $\alpha\in \mathbb{R}$,
   \begin{eqnarray*}
&& \mathbb{P}\Big(M_t\ge (\alpha+\beta \sum_{s\le t}|\triangle M_s|^2)x, \,\, \sum_{s\le t}|\triangle M_s|^2\ge v^2\mbox{~for some ~}t>0\Big )\\
&\le &   \exp\{-\frac{x^2}{4}(\alpha \beta +\frac{\beta^2v^2}{2})\}.
\end{eqnarray*}
\end{theorem}

   \section{ Application}\label{secE}
\setcounter{equation}{0}
\renewcommand{\theequation}{\thesection.\arabic{equation}}

In this section, we will derive  exponential  inequalities for block counting process in $\Lambda-$coalescents.  The $\Lambda-$coalescents  was introduced independently by Pitman \cite{p} and Sagitov \cite{s}. In this paper, the notation and details of $\Lambda-$coalescents  are from Limic and Talarczyk \cite{lt}.

Let $\Lambda$ be an probability measure on $[0,1]$, $\Pi=(\Pi_t)_{t\ge 0}$ is a Markov jump process.  $\Pi$ takes values in the set of partition of $\{1,2,\cdots\}$.  For any $n\ge 1$, the restriction $\Pi^n$ of $\Pi$ to $\{1,2,\cdots, n\}$ is a continuous time Markov chain with the following transitions: when $\Pi^n$ has $b$ blocks, any given $k-$tuples of blocks coalesces at rate 
                                       $$\lambda_{b,k}=\int_{0}^{1}r^{k-2}(1-r)^{b-k}\Lambda(dr)$$
where $2\le b\le n$. Let $N_t$ be the number of blocks of $\Pi_t$ at $t$.  In fact, $N=(N_t)_{t\ge 0}$ is a point process.   Limic and Talarczyk \cite{lt} presented  integral equation for $N$.  Define 
                 $$\pi(dt, dy, d{\bf x})=\sum_{k\ge 1}\varepsilon_{\{T_k, Y_k, {\bf X^k}\}}(dt, dy, d{\bf x})$$
where $\{{\bf X^k}\}$ is an independent array of i.i.d. random variables $(X_j^k)_{j,k\in \mathbb{N}}$ , where $X_j^k$ have uniform distribution on $[0,1]$.  The multivariate point processes $\pi$ have the compensator $dt\frac{\Lambda(dy)}{y^2}d{\bf x}$.

Limic and Talarczyk \cite{lt} obtained that 
            $$N_t=N_r-\int_{r}^{t}\int_{0}^{1}\int_{[0,1]^{\mathbb{N}}}f(N_{s-},y,{\bf x})\pi(ds, dy, d{\bf x})$$
for all $0<r<t$, where 
      $$f(k,y, {\bf x})=\sum_{j=1}^{k}1_{\{x_i\le y\}}-1+1_{\cap_{j=1}^{k}\{x_j>y\}}.$$
Define 
         $$\Psi(k)=\int_{0}^{1}\int_{[0,1]^{\mathbb{N}}}f(k,y,{\bf x})]\frac{\Lambda(dy)}{y^2}d{\bf x},$$  $$t=\int_{v_t}^{\infty}\frac{1}{\Psi(q)}dq,$$
and 
         $$M_t=\int_{0}^{t}\int_{0}^{1}\int_{[0,1]^{\mathbb{N}}}\frac{f(N_{s-},y,{\bf x})}{v_s}(\pi(dt, dy, d{\bf x})-ds\frac{\Lambda(dy)}{y^2}d{\bf x}).$$

$M=(M_t)_{t\ge 0}$ plays important role in the study of  $\Lambda-$coalescents. Limic and Talarczyk \cite{lt} obtained that $M$ is a square integrable martingale.  It is not difficult to see that $\triangle M\ge 0$, 
          $$\sum_{s\le t}|\triangle M_s|^2=\int_{0}^{t}\int_{0}^{1}\int_{[0,1]^{\mathbb{N}}}\frac{f^2(N_{s-},y,{\bf x})}{v^2_s}\pi(dt, dy, d{\bf x})$$
and 
          $$<M,M>_t=\int_{0}^{t}\int_{0}^{1}\int_{[0,1]^{\mathbb{N}}}\frac{f^2(N_{s-},y,{\bf x})}{v^2_s}ds\frac{\Lambda(dy)}{y^2}d{\bf x}.$$
We have the following result. 

  \begin{theorem} \label{mr6}Let  $M$ be  defined as above, we have 
    \begin{eqnarray*} \mathbb{P}\Big(M_t\ge x, \, \sum_{s\le t}|\triangle M_s|^2\le v^2\mbox{~for some ~}~t>0\Big)\le    \Big( \frac{v^2+x}{v^2}\Big)^{v^2}e^{-x} 
 \end{eqnarray*}
 and 
  \begin{eqnarray*}
&& \mathbb{P}\Big(M_t\ge (\alpha+\beta \sum_{s\le t}|\triangle M_s|^2)x, \,\, \sum_{s\le t}|\triangle M_s|^2\ge  <M,M>_t+v^2\mbox{~for some ~}t>0\Big )\\
&\le &   \exp\{-\frac{x^2}{2}(\alpha \beta +\frac{\beta^2v^2}{2})\}.
\end{eqnarray*}
where  $x\ge 0, \beta>0, v>0$, $\alpha\in \mathbb{R}$.

   \end{theorem}

\section *{Acknowledgments}

This work  was supported by  National Key R$\&$D Program of China (No.2018YFA0703900) and Shandong Provincial Natural Science Foundation (No. ZR2019ZD41).

%\section*{References}
%\bibliographystyle{elsarticle-num}
%\bibliographystyle{abbrv}
%\bibliographystyle{abbrv}
%\bibliography{mybibfilenew}

\end{document}